\documentclass[11pt]{article}
\usepackage{fullpage}
\usepackage{amsmath,amsfonts,amsthm,mathrsfs,mathpazo,xspace,hyperref,graphicx}

\newtheorem{theorem}{Theorem}

\newtheorem{lemma}[theorem]{Lemma}

\newtheorem{claim}[theorem]{Claim}

\newtheorem{corollary}[theorem]{Corollary}

\newcommand{\beq}{\begin{equation}}
\newcommand{\eeq}{\end{equation}}

\newcommand{\C}{\ensuremath{\mathbb{C}}}

\newcommand{\R}{\ensuremath{\mathbb{R}}}

\newcommand{\Tr}{\mbox{\rm Tr}}
\newcommand{\Id}{\ensuremath{\mathop{\rm Id\,}\nolimits}}
\newcommand{\exactness}{\mbox{\rm ex}}
\newcommand{\rc}{\eta}

\newcommand{\os}[1]{\|#1\|_{os}}
\newcommand{\nc}[1]{\|#1\|_{nc}}

\newcommand{\cA}{\mathcal{A}}
\newcommand{\cB}{\mathcal{B}}

\newcommand{\eps}{\varepsilon}

\newif\ifnotes\notesfalse

\ifnotes
\usepackage{color}
\definecolor{mygrey}{gray}{0.50}
\newcommand{\notename}[2]{{\textcolor{mygrey}{\footnotesize{\bf (#1:} {#2}{\bf ) }}}}
\newcommand{\noteswarning}{{\begin{center} {\Large WARNING: NOTES ON}\end{center}}}

\else

\newcommand{\notename}[2]{{}}
\newcommand{\noteswarning}{{}}

\fi

\begin{document}

\title{\bf Elementary Proofs of Grothendieck Theorems for Completely Bounded Norms}
\author{Oded Regev \footnote{CNRS, D{\'e}partement d'Informatique, {\'E}cole normale sup{\'e}rieure, Paris and Blavatnik School of Computer Science, Tel Aviv University.
Supported by a European Research Council (ERC) Starting Grant.
   } \and Thomas Vidick \footnote{Computer Science and Artificial Intelligence Laboratory, Massachusetts Institute of Technology. Supported by the National Science Foundation under Grant No. 0844626.}}
\date{}
\maketitle
\noteswarning

\begin{abstract}
We provide alternative proofs of two recent Grothendieck theorems for jointly completely bounded bilinear forms, 
originally due to Pisier and Shlyakhtenko~\cite{PS02OSGT} and Haagerup and Musat~\cite{HM08}. Our proofs are elementary and are inspired by the so-called embezzlement states in quantum information theory. 
Moreover, our proofs lead to quantitative estimates.
\end{abstract}

\section{Introduction}

Published in 1953, Grothendieck's theorem~\cite{Gro53}, a non-trivial statement regarding bounded bilinear forms on $L_\infty\times L_\infty$, had a major impact on Banach space theory. 
A non-commutative extension of Grothendieck's theorem to the setting of bounded bilinear forms on $C^*$-algebras,
already conjectured in~\cite{Gro53}, was first proved by Pisier under some approximability assumption~\cite{Pisier78NCGT}, and then in full generality by Haagerup~\cite{Haagerup85NCGT}.
More recently, analogues of Grothendieck's theorem for jointly completely bounded bilinear forms were obtained by Pisier and Shlyakhtenko~\cite{PS02OSGT} and by Haagerup and Musat~\cite{HM08}.
The former holds for forms defined on exact operator spaces (see also~\cite[Section 18]{PisierGT} for an alternative proof by Pisier and de la Salle)
and the latter holds for forms defined on arbitrary $C^*$-algebras. Such statements were earlier conjectured by Effros and Ruan~\cite{ER91} and by Blecher~\cite{Blecher92}. We refer the reader to~\cite{PisierGT} for a comprehensive survey of Grothendieck's theorem and its extensions. 

The purpose of this note is to give new, simpler (in our opinion), and more quantitative proofs of these two recent results. 
The existing proofs crucially use a kind of non-commutative probability space defined on Type III von Neumann algebras and are somewhat elaborate. In contrast, our proof technique, based on ideas originating in quantum information theory, is much more elementary and explicit. Our proof also leads to more quantitative versions of these Grothendieck theorems, which may be useful in some applications. (See~\cite{RegevV12a} for an application to quantum multiplayer games.) 

Similarly to~\cite{HM08} and the proof by Pisier and de la Salle~\cite[Section 18]{PisierGT}, our proof 
is based on a transformation which reduces the question to one of the better-understood non-commutative versions of Grothendieck's 
theorem~\cite{Haagerup85NCGT,JungeP95}. Our transformation is much more concrete, and is described in detail
in our main theorem, stated next.

\begin{theorem}\label{thm:os-gt} Let $\mathcal{A},\mathcal{B}$ be $C^*$-algebras, $E\subseteq \cA$, $F\subseteq \cB$ operator spaces, and $u:E\times F\to \C$ a bilinear form. Let $d\geq 1$ be an integer and $M_d$ the space of $d\times d$ complex matrices. There exists a unit vector $\Phi\in \C^d\otimes \C^d$, with associated bilinear form $\phi$ defined on $M_d\times M_d$ by $\phi(a,b) = \langle \Phi,(a\otimes b) \Phi\rangle$, such that for any 
finite sequences $(x_i)_i$ in $E$, $(y_i)_i$ in $F$, and positive reals $(t_i)_i$ there exist finite sequences $(\tilde{x}_j)_j$ in $E\otimes M_d$ and $(\tilde{y}_j)_j$ in $F\otimes M_d$ satisfying
\begin{align}
\Big\|\sum_{j}\, \tilde{x}_{j} \tilde{x}_{j}^* \Big\| \,\leq\, \Big\|\sum_i \,x_ix_i^*\Big\|,& \qquad\Big\|\sum_{j}\, \tilde{x}_{j}^*\tilde{x}_{j}  \Big\| \,\leq\,\Big\|\sum_i t_i^2\, {x}_i^*{x}_i \Big\|, \notag\\
 \Big\|\sum_{j}\, \tilde{y}_{j} \tilde{y}_{j}^* \Big\| \,\leq\,\Big\|\sum_i t_i^{-2} \,{y}_i{y}_i^*\Big\|,&\qquad \Big\|\sum_{j}\, \tilde{y}_{j}^*\tilde{y}_{j}  \Big\| \,\leq\, \Big\|\sum_i \,{y}_i^*{y}_i\Big\|,\label{eq:os-gt-cons}
\end{align}
and such that
\begin{equation}\label{eq:os-gt-main}
\Big| \sum_{j} \,(u\otimes \phi)(\tilde{x}_{j},\tilde{y}_{j}) \Big| \,\geq\, \Big| \sum_i \,u({x}_i, {y}_i )\Big| \, - C\,\frac{\ln\big(1+\max_i\{t_i,t_i^{-1}\}\big)}{1+\ln d} \sum_i | u({x}_i, {y}_i ) |,
\end{equation}
where $C>0$ is a universal constant. 
\end{theorem}

We stress that both the vector $\Phi$ and the mapping $(x_i,y_i,t_i)\mapsto(\tilde{x}_j,\tilde{y}_j)$ are explicit.
In particular, the vector $\Phi=\Phi_d$ whose existence is promised in the theorem is known as the ``embezzlement state''~\cite{vDH03} in quantum information theory, and is defined as 
\begin{equation}\label{eq:def-phi}
 \Phi_d \,:=\, Z_d^{-1/2} \,\sum_{i=1}^d \, \frac{1}{\sqrt{i}}\, e_i\otimes e_i\,\in\C^d\otimes\C^d,
\end{equation}
where $(e_i)$ is the canonical basis of $\C^d$ and $Z_d = \sum_{i=1}^d i^{-1}$ the proper normalization constant. 
As an aside, we note that the name ``embezzlement" comes from an intriguing property that such states possess: any entangled state can be ``distilled'' from $\Phi_d$ (assuming $d$ large enough) using local operations while keeping $\Phi_d$ essentially intact. This property implies, for instance, that in the definition of the jointly completely bounded norm (see~\eqref{eq:defjcb}) it suffices to consider only evaluations of the amplified bilinear form on the states $\Phi$.
The construction of $(\tilde{x}_j,\tilde{y}_j)$ is also explicit, and relies on the construction of a family of $d$-dimensional ``line'' matrices given in Claim~\ref{claim:line-matrices} (see also Figure~\ref{fig:matrices} for an illustration).  It is the specific interplay between these matrices and the state $\Phi$ that guarantees the validity of~\eqref{eq:os-gt-cons} and~\eqref{eq:os-gt-main}.

\paragraph{Organization of the paper.} We present the proof of Theorem~\ref{thm:os-gt} in Section~\ref{sec:main}. In Section~\ref{sec:grothendieck} we apply the theorem to derive short proofs of the main results of~\cite{HM08} (in Section~\ref{sec:hm}) and of~\cite{PS02OSGT} (in Section~\ref{sec:ps}). We also obtain new quantitative estimates for both results.
 
\paragraph{Acknowledgments.} We thank Gilles Pisier for allowing us to include Claim~\ref{clm:oh}. We also thank him and Carlos Palazuelos for useful comments.

\section{Proof of Theorem~\ref{thm:os-gt}}\label{sec:main}

The main tool in our proof of Theorem~\ref{thm:os-gt} is the construction of a special family of ``line'' matrices (see
Figure~\ref{fig:matrices} for an illustration). We note that the choice of these matrices can be shown to be optimal in a certain precise sense.

\begin{figure}
\begin{centering}
\ifpdf
\includegraphics[width=0.4\textwidth]{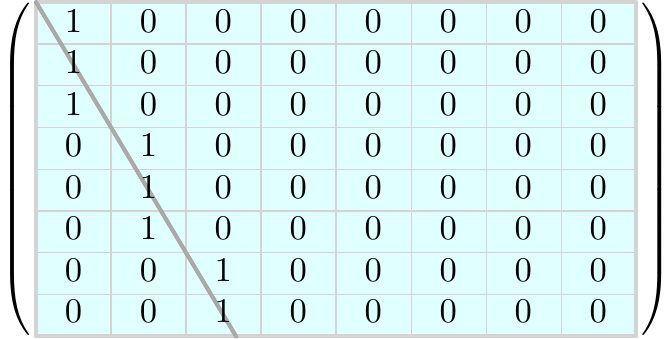} \qquad
\includegraphics[width=0.4\textwidth]{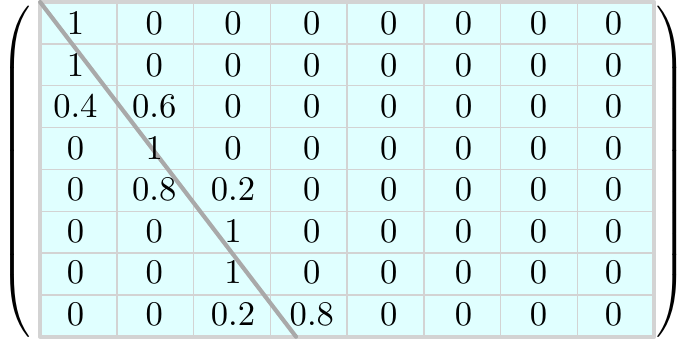}
\fi
\caption{$L(\sqrt{3})$ (left) and $L(\sqrt{2.4})$ (right) for $d=8$.}
\label{fig:matrices}
\end{centering}
\end{figure}

\begin{claim}\label{claim:line-matrices} 
For any integer $d\geq 1$ there exists a collection of $d\times d$ matrices $\big\{L(t)\big\}_{t\in\R_+}$, parametrized by the positive reals, satisfying the following conditions:
\begin{enumerate}
\item For all $t>0$, $L(t)$ has non-negative entries that sum to at most $1$ in every row, and to at most $t^2$ in every column;
\item There exists a unit vector $z \in \R^d$ with non-negative entries such that for all $t>0$,
\begin{align*}
\Big(1-C\frac{\ln(1+\max(t,t^{-1}))}{1+\ln d}\Big)t \,\leq\,\langle z,L(t)z\rangle \,\leq\,t,
\end{align*}
where $C>0$ is a universal constant. In fact, one can take the unit vector $z = Z_d^{-1/2} (i^{-1/2})_{i=1}^{d}$, where 
$Z_d\,=\, \sum_{i=1}^d \,\frac{1}{i} \,\leq\, 1 + \int_1^d \,\frac{1}{r} \text{dr} \,=\, 1+\ln d$
is the proper normalization constant. 
\end{enumerate}
\end{claim}

\begin{proof}
Let $t$ be a positive real, and define $L(t)$ by setting its $(i,j)$-th entry $L(t)_{i,j}$, for $i,j\in \{1,\ldots,d\}$, to the length of the interval $ [i-1,i) \cap [(j-1)t^2,jt^2)$. The first item in the claim clearly holds. 
For the second, we start with the upper bound, which actually holds for any unit vector $z$ and any $L(t)$ satisfying the constraints in the first item. Indeed, applying the Cauchy-Schwarz inequality,
$$\big|\langle z, L(t)z\rangle\big| = \Big|\sum_{i,j} L(t)_{i,j} \,z_{i}^* z_j \Big| \leq \Big( \sum_{i,j} L(t)_{i,j}|z_i|^2  \Big)^{1/2}\Big( \sum_{i,j} L(t)_{i,j} |z_j|^2 \Big)^{1/2} \leq t. $$

It remains to prove the lower bound. Using the vector $z$ appearing in the statement of the claim,
\begin{align*}
\langle z, L(t)z\rangle &= \frac{1}{Z_d}\int_{0}^{d\min(1,t^2)} \frac{1}{\sqrt{\lceil r/t^2\rceil \lceil r \rceil}}\, \text{dr}\\
&\geq \frac{1}{Z_d}\int_{0}^{d\min(1,t^2)} \frac{1}{\sqrt{( 1+r/t^2)( 1+r)}}\, \text{dr}\\
& = \frac{2t}{Z_d}\Big(\ln\Big(\sqrt{d\min(1,t^2)+1}+\sqrt{d\min(1,t^2)+t^2}\Big) - \ln\big(t+1\big)\Big)\\
&\geq t\, \frac{\ln\big( 2d\min(1,t^2)+(1+t)^2\big) - 2\ln\big(t+1\big)}{1+\ln d}\\
&\geq t\,\Big( 1- C\frac{\ln(1+\max(t,t^{-1}))}{1+\ln d} \Big),
\end{align*}
for some universal constant $C>0$. 
\end{proof}

The matrices constructed in the previous claim let us show the following lemma, which provides the key estimates required for the proof of Theorem~\ref{thm:os-gt}. 

\begin{lemma}\label{lem:lines} 
For any integer $d\geq 1$ and real $t>0$ there exists a sequence $(L^{r}(t))_r$ of $d^2$ matrices of dimensions $d\times d$, satisfying the following conditions for all $t>0$: 
\begin{align}
 \sum_r L^r(t) L^r(t)^*  \,\leq\, \Id\quad\text{and}\quad   \sum_r L^r(t)^* L^r(t)  \,\leq\, t^2\Id,\label{eq:line-bounds}\\
\Big|\sum_{r} \langle \Phi,(L^{r}(t)\otimes L^{r}(t)) \Phi\rangle -t\Big| \,\leq\, C\,t\,\frac{\ln(1+\max(t,t^{-1}))}{1+ \ln d},\label{eq:line-emb}
\end{align}
where $C>0$ is a universal constant and $\Phi \in \C^d \otimes \C^d$ is the unit vector defined in~\eqref{eq:def-phi}.
\end{lemma}

\begin{proof}
Let $(L(t))_{t\in \R_+}$ be the collection of matrices whose existence is promised by Claim~\ref{claim:line-matrices},  $z$ the corresponding vector, and note that $\Phi = \sum_i z_i\, e_i\otimes e_i$, where $(e_i)$ is the canonical basis of $\C^d$. For $1\leq i,j \leq d$ define $L^{i+(j-1)d}(t)$ by setting its $(i,j)$-th entry to $(L(t)_{i,j})^{1/2}$, and all other entries to $0$. Then $\sum_r L^r(t) L^r(t)^*$ is a diagonal matrix whose $(i,i)$-th entry is the sum of the entries in the $i$-th row of $L(t)$, while $\sum_r L^r(t)^* L^r(t) $ is diagonal with $(j,j)$-th entry the sum of the entries in the $j$-th column of $L(t)$. Hence the constraints~\eqref{eq:line-bounds} are satisfied as a consequence of Item~1 from Claim~\ref{claim:line-matrices}. The condition~\eqref{eq:line-emb} follows immediately from Item~2 of Claim~\ref{claim:line-matrices} by noting that
\[\sum_{r} \langle \Phi, ( L^{r}(t)\otimes L^{r}(t) )\Phi \rangle \,=\, \langle z, L(t) z \rangle.\qedhere\]
\end{proof}

Given Lemma~\ref{lem:lines}, the proof of Theorem~\ref{thm:os-gt} is relatively straightforward, and we give it below.

\begin{proof}[Proof of Theorem~\ref{thm:os-gt}]
Consider finite sequences $(x_i)_i$ in $E$, $(y_i)_i$ in $F$, positive reals $(t_i)_i$, and let $d$ be a positive integer. Let $(L^{r}(t_i))_{r\geq 1}$ be the matrices constructed in Lemma~\ref{lem:lines}. For each pair $(i,r)$ define 
\begin{align*}
\tilde{x}_{i,r} :=  {x}_i \otimes L^r(t_i) \,\in\, E\otimes M_d \quad\text{and}\quad
\tilde{y}_{i,r} := t_i^{-1} \,{y}_i \otimes L^r(t_i) \,\in\, F\otimes M_d.
\end{align*}
The bounds in~\eqref{eq:line-bounds} directly lead to the following upper bounds:
\begin{align*}
\Big\|\sum_{i,r}\, \tilde{x}_{i,r} \tilde{x}_{i,r}^* \Big\| &=\Big\|\sum_i\sum_{r} x_ix_i^* \otimes L^r(t_i)L^r(t_i)^*\Big\|\,\leq\, \Big\|\sum_i \,x_ix_i^*\Big\|,\\
\Big\|\sum_{i,r}\, \tilde{x}_{i,r}^*\tilde{x}_{i,r}  \Big\| &=\Big\|\sum_i\sum_{r} {x}_i^*{x}_i \otimes L^r(t_i)^* L^r(t_i)\Big\|\,\leq\,\Big\|\sum_i t_i^2\, {x}_i^*{x}_i \Big\|,
\end{align*}
and
\begin{align*}
 \Big\|\sum_{i,r}\, \tilde{y}_{i,r} \tilde{y}_{i,r}^* \Big\| &= \Big\|\sum_i\sum_{r} t_i^{-2}\,y_iy_i^* \otimes L^r(t_i)L^r(t_i)^*\Big\|\,\leq\,\Big\|\sum_i t_i^{-2} \,{y}_i{y}_i^*\Big\|,\\
\Big\|\sum_{i,r}\, \tilde{y}_{i,r}^*\tilde{y}_{i,r}  \Big\| &= \Big\|\sum_i\sum_{r} t_i^{-2}\,{y}_i^*{y}_i \otimes L^r(t_i)^* L^r(t_i)\Big\|\,\leq\, \Big\|\sum_i \,{y}_i^*{y}_i\Big\|,
\end{align*}
proving~\eqref{eq:os-gt-cons}. To conclude it remains to evaluate
\begin{align*}
\Big| \sum_{i,r} \,(u\otimes \phi)(\tilde{x}_{i,r},\tilde{y}_{i,r}) \Big| 
&= \Big| \sum_{i,r}\,t_i^{-1}\,u({x}_i, {y}_i )\,\langle \Phi,( L^r(t_i) \otimes L^r(t_i)) \Phi\rangle \Big| \notag\\
&\geq \Big| \sum_i \,u({x}_i, {y}_i )\Big| \, - C\,\frac{\ln\big(1+\max_i\{t_i,t_i^{-1}\}\big)}{1+\ln d} \sum_i  | u({x}_i, {y}_i ) |,
\end{align*}
where the inequality follows from~\eqref{eq:line-emb}. 
\end{proof}

\section{Two Grothendieck theorems}\label{sec:grothendieck}

In this section we show how the main results of~\cite{HM08} and~\cite{PS02OSGT}, as well as new quantitative estimates, can be derived from Theorem~\ref{thm:os-gt}. We first recall some useful definitions and notation, and refer the reader to~\cite{pisierbook} for additional background on operator spaces.

\paragraph{Norms on bilinear forms.}
Let $\cA$, $\cB$ be $C^*$-algebras, and $E\subseteq \cA$, $F\subseteq \cB$ operator spaces. A bilinear form $u:E\times F \to \C$ is called \emph{jointly completely bounded} if the naturally associated map $\tilde{u}:E\to F^*$ is completely bounded. In more detail, we define
\begin{align}\label{eq:defjcb}
 \|u\|_{jcb} := \sup_d \|u_d\|,
\end{align}
where for any integer $d\geq 1$, $u_d$ is the amplification 
\begin{align*}
u_d:\quad& E\otimes_{\min} M_d \times F\otimes_{\min} M_d \,\to\, M_d \otimes_{\min} M_d\\[2.5mm]
& \big(\sum a_i \otimes x_i,\sum b_i\otimes y_i\big)\,\mapsto\, \sum_{i,j} u(a_i,b_j) \, x_i\otimes y_j.
\end{align*}
For any unit vector $\Omega\in\C^d\otimes \C^d$ we also consider an associated map $u_d^\Omega$, defined as 
\begin{align}
u_d^\Omega:\quad& E\otimes_{\min} M_d \times F\otimes_{\min} M_d \,\to\, \C\notag\\[2.5mm]
& \hskip1.5cm\big(a,\,b\big)\hskip1.6cm\mapsto\, \langle \Omega,\,u_d(a,b) \Omega\rangle.\label{eq:def-vd}
\end{align}
Clearly for any integer $d$ and unit vector $\Omega$ it holds that 
\begin{equation}\label{eq:uphi-norm}
\|u_d^\Omega\|_{jcb} \,=\, \|u\|_{jcb},
\end{equation}
and in fact for any integer $n$ we have $\|u_n\|\leq \|(u_d^\Omega)_n\| \leq \|u_{dn}\|$. We will also make use of the notion of \emph{tracially bounded} bilinear forms, which first appears in~\cite{Blecher89}. It can be defined by specializing $\Omega$ in~\eqref{eq:def-vd} to the vectors $\Psi$ (known as the ``maximally entangled states" in quantum information theory),
\begin{equation}\label{eq:def-psi}
\Psi_d \,:=\, d^{-1/2}\, \sum_{i=1}^d \,e_i\otimes e_i\,\in\C^d\otimes\C^d.
\end{equation}
In detail, a bilinear map $u$ is said to be tracially bounded if the following supremum is finite,
$$ \|u\|_{tb}\,:=\, \sup_d \|u_d^{\Psi}\| \,=\, \sup \Big| \sum_{i,j}\,u(a_i,b_j)\,\langle \Psi,(x_i\otimes y_j)\,\Psi\rangle \big|\,=\,\sup \Big| \sum_{i,j} d^{-1}\Tr(x_iy_j^t)\,u(a_i,b_j) \big|,$$
where the supremum is taken over all integers $d\geq 1$ and $\sum a_i\otimes x_i \in E \otimes_{\min} M_d$, $\sum b_i\otimes y_i \in F \otimes_{\min} M_d$ of norm at most $1$. We clearly have $\|u\| \le \|u\|_{tb} \le \|u\|_{jcb}$.

\paragraph{Grothendieck values associated with bilinear forms.}
Grothendieck's theorem and its extensions can be stated in a number of essentially equivalent ways. The formulations we use here are in the form of an inequality that involves the following quantity:  
\begin{equation}\label{eq:sdp}
\os{u} \,:=\,\sup ~ \Big|\sum_i\, u(x_i,y_i)\Big|,
\end{equation}
where the supremum is taken over all finite sequences $(x_i)_{i}$ in $E$, $(y_i)_{i}$ in $F$, and positive reals $(t_i)_{i}$ satisfying the constraint\footnote{It is easy to see that we could equivalently use the constraint $\|\sum_i x_ix_i^*\|^{1/2} \|\sum_i y_i^*y_i\|^{1/2} + \|\sum_i t_i^2 x_i^*x_i\|^{1/2} \|\sum_i t_i^{-2}y_i y_i^*\|^{1/2} \leq 2$ instead of~\eqref{eq:sdpcons}. This is
the way it appears in, e.g.,~\cite[Theorem 0.4]{PS02OSGT}.}
\begin{align}
\max\Big\{\,\Big\| \sum_i  \,x_i x_i^* \Big\| + \Big\|\sum_i  t_i^2\, x_i^* x_i \Big\|, \, \Big\|\sum_i  t_i^{-2}\,y_i y_i^* \Big\| + \Big\|\sum_i  \,y_i^* y_i \Big\|\,\Big\}\,\leq\,2.\label{eq:sdpcons}
\end{align}
If we further restrict the coefficients $(t_i)$ to $t_i=1$ for all $i$, then we use $\nc{u}$ to denote the resulting supremum in~\eqref{eq:sdp}. Clearly $\|u\|\leq \nc{u}\leq\os{u}$. Our choice of normalization for the constraint~\eqref{eq:sdpcons} differs from the one adopted in~\cite{PS02OSGT,HM08}, where the constant $2$ on the right-hand side is replaced by a $1$. With our normalization, the following inequalities are easily seen to hold (see Appendix~\ref{app:ub-norms} for the proof): 
\begin{equation}\label{eq:relaxation}
  \|u\|_{tb} \,\leq\, \nc{u}\qquad\text{and}\qquad  \|u\|_{jcb} \,\leq\, \os{u}.
\end{equation}

\paragraph{Row and column norms.}
In order to state our quantitative estimates, for any operator space $E\subseteq\mathcal{A}$ we define a quantity $\rc(E)$ as
$$\rc(E)\,:=\, \max\Big\{\, \sup_{(x_i):\, \|\sum_i x_i^* x_i \|\leq 1}\, \big\|\,\sum_i\, x_i x_i^* \,\big\|^{1/2},\, \sup_{(x_i):\, \|\sum_i x_i x_i^* \|\leq 1}\, \big\|\,\sum_i\, x_i^* x_i \,\big\|^{1/2} \,\Big\}.\footnote{In other words, $\rc(E)$ is the maximum of the norms of the natural maps $C\otimes_{\min} E \to R\otimes_{\min} E$ and $R\otimes_{\min} E \to C\otimes_{\min} E$ (as maps
between Banach spaces).}$$
It is not hard to see that $\rc(M_n)\leq \sqrt{n}$; see Claims~\ref{clm:rcbound} and~\ref{clm:oh} in Appendix~\ref{sec:eta-bound} for a proof and for other upper bounds on $\rc$.

\subsection{Forms on \texorpdfstring{$C^*$}{C*}-algebras}\label{sec:hm}

In this section we prove the following corollary of Theorem~\ref{thm:os-gt}, reproving the main result of Haagerup and Musat~\cite{HM08} and obtaining new quantitative estimates.  

\begin{corollary}\label{cor:os-gt} Let $\mathcal{A},\mathcal{B}$ be $C^*$-algebras, and $u:\mathcal{A}\times \mathcal{B}\to \C$ a jointly completely bounded bilinear form. Then 
\begin{equation}\label{eq:cor-os-0}
 \|u\|_{jcb}\,\leq\,\os{ u } \,\leq\, 2\| u\|_{jcb}.
\end{equation}
Moreover, if $\rc(\mathcal{A}),\rc(\mathcal{B})$ are finite then for any $\eps>0$ and any $d \ge (2\rc(\mathcal{A})\rc(\mathcal{B})/\eps)^{C/\eps}$, where $C>0$ is a universal constant,
$$ (1-\eps) \os{ u } \,\leq\, 2\|u_d^\Phi\| \,\leq\, 2\| u_d\|.$$
\end{corollary}

To prove the corollary we will use Theorem~\ref{thm:os-gt} to perform a reduction to the ``non-commutative Grothendieck theorem''~\cite{Haagerup85NCGT} which shows that an inequality similar to~\eqref{eq:cor-os-0} holds for the case of \emph{bounded} forms defined on $C^*$-algebras.

\begin{theorem}[Non-commutative GT,~\cite{Haagerup85NCGT}]\label{thm:ncgt}
Let $\cA,\cB$ be $C^*$-algebras and $u:\cA\times \cB \to \C$ a bounded bilinear form. Then
\begin{align*}
\nc{u}\,\leq\, 2 \|u\|.
\end{align*}
\end{theorem}

\begin{proof}[Proof of Corollary~\ref{cor:os-gt}]
The first inequality is~\eqref{eq:relaxation}. For the second inequality, let $\eps>0$ and $(x_i,y_i,t_i)_i$ finite sequences satisfying~\eqref{eq:sdpcons} and such that 
\begin{equation}\label{eq:os-val}
\Big|\sum_i u(x_i,y_i)\Big| \geq (1-\eps) \os{u}.
\end{equation}
By Theorem~\ref{thm:os-gt} for any $d$ there exists a unit vector $\Phi\in\C^d\otimes \C^d$ and sequences $(\tilde{x}_j)$, $(\tilde{y}_j)$ such that 
\begin{align*}
\Big| \sum_{j} \,u_d^{\Phi}(\tilde{x}_{j},\tilde{y}_{j}) \Big| &\geq \Big| \sum_i \,u({x}_i, {y}_i )\Big| \, - C\,\frac{\ln\big(1+\max_i\{t_i,t_i^{-1}\}\big)}{1+\ln d} \sum_i | u({x}_i, {y}_i ) |\\
&\geq \Big(1-\eps- C\,\frac{\ln\big(1+\max_i\{t_i,t_i^{-1}\}\big)}{1+\ln d}\Big)\,\os{u},
\end{align*}
where for the second inequality we use~\eqref{eq:os-val} and observe that for any numbers $\alpha_i$ of modulus $1$,
$(\alpha_i x_i,y_i,t_i)$ satisfies~\eqref{eq:sdpcons} and hence $\sum_i | u({x}_i, {y}_i ) | \le \os{u}$. By choosing $d \ge (1+\max_i\{t_i,t_i^{-1}\})^{C / \eps}$ we obtain
\begin{equation*}
(1-2\eps)\os{u}\,\leq\,\nc{u_d^\Phi} \leq 2\|u_d^\Phi\| \leq 2\|u\|_{jcb},
\end{equation*}
 where the first inequality holds since by~\eqref{eq:os-gt-cons} the $(\tilde{x}_i,\tilde{y}_i,t_i=1)$ satisfy~\eqref{eq:sdpcons}, the second inequality follows from Theorem~\ref{thm:ncgt}, and the third inequality follows from~\eqref{eq:uphi-norm}. Letting $\eps\to 0$ proves the second inequality in~\eqref{eq:cor-os-0}. 

For the ``moreover" part of the corollary, Claim~\ref{claim:maxt} below (with $E=\mathcal{A}$ and $F=\mathcal{B}$) shows that we can choose the sequence $(x_i,y_i,t_i)_{i}$ in a way that $\max_i \big\{ t_i,t_i^{-1} \big\} \leq 8\,\rc(\mathcal{A})\rc(\mathcal{B})/\eps$, which, together with the bound on $d$ shown above, leads to the estimate claimed in the corollary.  
\end{proof}

\begin{claim}\label{claim:maxt} 
Let $E\subseteq \mathcal{A}$, $F\subseteq \mathcal{B}$ be operator spaces such that $\rc(E),\rc(F)<\infty$. For any $u:E \times F \to\C$ and any $\eps>0$ there exists $(x_i,y_i,t_i)$ satisfying~\eqref{eq:sdpcons} such that $\max_i \big\{ t_i,t_i^{-1} \big\} \leq 8\rc(E)\rc(F)/\eps$ and 
$$ \Big|\sum_i u(x_i,y_i)\Big| \, \geq\, (1-\eps) \os{u}.$$
\end{claim}

\begin{proof} 
Let $(x_i,y_i,t_i)$ be a sequence satisfying the constraint~\eqref{eq:sdpcons} and such that
$$ \Big| \sum_i u(x_i,y_i) \Big| \, \geq\, (1-\eps/2) \os{u}.$$
Let $T = 8\, \rc(E)\rc(F)/\eps >1$, and define $S_1 = \{i:\, t_i \geq T\}$ and $S_2 = \{i:\, t_i^{-1} \geq T\}$. Note that $S_1$ and $S_2$ are disjoint, and let $S = S_1\cup S_2$. For every $i\in S_1$ (resp. $i\in S_2$) let $\tilde{x}_i = T\,x_i /(2\rc(E))$ and $\tilde{y}_i = y_i/(2\rc(F))$ (resp. $\tilde{x}_i = x_i/(2\rc(E))$ and $\tilde{y}_i = T\, y_i/(2\rc(F))$). We have
$$ \Big\| \sum_{i\in S_1} \tilde{x}_i\tilde{x}_i^* \Big\| + \Big\| \sum_{i\in S_1} \tilde{x}_i^*\tilde{x}_i \Big\|\,\leq\, \frac{T^2}{2} \Big\| \sum_{i\in S_1} x_i^* x_i\Big\| \,\leq\, 1,$$
where for the first inequality we used the definition of $\rc(E)$ to upper bound the first term, and the second inequality follows from the constraint~\eqref{eq:sdpcons} and the definition of $S_1$. Similarly,
$$ \Big\| \sum_{i\in S_1} \tilde{y}_i\tilde{y}_i^* \Big\| + \Big\| \sum_{i\in S_1} \tilde{y}_i^*\tilde{y}_i \Big\|\,\leq\, \frac{1}{2} \Big\| \sum_{i\in S_1} y_i^* y_i\Big\|\,\leq\, 1$$
by~\eqref{eq:sdpcons}, and similar inequalities hold for $S_2$. Together these bounds imply that $(\tilde{x}_i,\tilde{y}_i,\tilde{t}_i=1)_{i\in S}$ satisfies~\eqref{eq:sdpcons}. Hence it must be that
$$ \Big| \sum_{i\in S} u(x_i,y_i) \Big|\,=\, 4\frac{\rc(E)\rc(F)}{T}\, \Big| \sum_{i\in S} u(\tilde{x}_i,\tilde{y}_i) \Big| \,\leq \, (\eps/2)\, \nc{u}\,\leq\, (\eps/2)\, \os{u},$$ 
where the first inequality uses the definition of $T$. Hence
$$ \Big|\sum_{i\notin S} u(x_i,y_i)\Big| \, \geq\, \Big| \sum_{i} u(x_i,y_i) \Big| -(\eps/2) \os{u}\,\geq\, (1-\eps)\os{u},$$
which proves the claim by restricting the initial sequence $(x_i,y_i,t_i)$ to those $i\notin S$.  
\end{proof}

\subsection{Forms on exact operator spaces}\label{sec:ps}

Our second corollary applies to completely bounded forms defined on operator spaces that are exact. This reproves the main result of~\cite{PS02OSGT}. As before, we also obtain a new quantitative estimate. To state the corollary, following~\cite[Section~16]{PisierGT} for a finite-dimensional operator space $E$ and integer $n$ we define
$$ \exactness_n(E)\,:=\, \inf\big\{ d_{cb}(E,F)\,|\, F\subseteq M_n\big\},$$
where $d_{cb}(E,F)$ is defined as the infimum of $\|v\|_{cb}\|v^{-1}\|_{cb}$ over all isomorphisms $v:E\to F$, and 
$$\exactness(E) \,:=\, \sup\big\{ \inf_n \exactness_n(E_1)\,|\, E_1\subseteq E,\,\dim(E_1)<\infty\big\}.$$

\begin{corollary}\label{cor:exact-gt}
Let $\mathcal{A},\mathcal{B}$ be $C^*$-algebras, $E\subseteq \cA$, $F\subseteq \cB$ operator spaces, and $u:E\times F\to \C$ a jointly completely bounded bilinear form. Then 
\begin{equation}\label{eq:exact-gt}
\|u\|_{jcb}\,\leq\,\os{ u } \,\leq\, 4\exactness(E)\exactness(F)\,\| u\|_{jcb}.
\end{equation}
Moreover, if $E,F$ are finite dimensional then for any $\eps>0$, $n \ge 1$, $d \ge (2\rc(E)\rc(F)/\eps)^{C/\eps}$ and $d'\geq C'\eps^{-2}\ln (nd)$, where $C,C'>0$ are universal constants, 
$$ (1-\eps) \os{ u } \,\leq\, 4 \exactness_n(E)\exactness_n(F)\,\| u_{dd'}^{\Phi\otimes\Psi} \|\,\leq\, 4 \exactness_n(E)\exactness_n(F)\,\|u_{d d'}\|,$$
where $\Phi = \Phi_d$, $\Psi = \Psi_{d'}$ are as defined in~\eqref{eq:def-phi} and~\eqref{eq:def-psi} respectively.
\end{corollary}

We note that the result from~\cite{PS02OSGT} is in fact slightly stronger, as it proves that inequality~\eqref{eq:exact-gt} still holds for a variant of $\os{u}$ in which the constraint~\eqref{eq:sdpcons} is replaced by the potentially looser constraint 
\begin{equation}\label{eq:l1-cons}
\max\Big\{\,\Big\| \sum_i  \,x_i x_i^* \Big\|^{1/2} + \Big\|\sum_i  t_i^2\, x_i^* x_i \Big\|^{1/2}, \, \Big\|\sum_i  t_i^{-2}\,y_i y_i^* \Big\|^{1/2} + \Big\|\sum_i  \,y_i^* y_i \Big\|^{1/2}\,\Big\}\,\leq\,2.
\end{equation}
Corollary~\ref{cor:exact-gt} (including the quantitative estimate) also holds in this stronger form, as follows from a straightforward modification of the proof. The main observation is that Theorem~\ref{thm:os-gt} operates on each of the four terms in~\eqref{eq:sdpcons} separately, and hence applies equally well to the modified constraint~\ref{eq:l1-cons}. For convenience we prove the corollary in the form stated above. 

To prove Corollary~\ref{cor:exact-gt} we will use Theorem~\ref{thm:os-gt} to perform a reduction to a Grothendieck inequality due to Junge and Pisier~\cite{JungeP95} which applies to the case of tracially bounded bilinear forms. 
We state the main result in~\cite{JungeP95} as it appears in~\cite[Section 16]{PisierGT} where an alternative proof is given (based on~\cite{HT99}).  The ``moreover'' part of the theorem follows from that alternative proof, and we include the proof in Appendix~\ref{apx:quantitative}.

\begin{theorem}[\cite{JungeP95}]\label{thm:jungepisier}
For any tracially bounded bilinear form $u:E\times F \to \C$ on exact operator spaces,
\begin{equation}\label{eq:ps}
\nc{u}\,\leq\, 4\exactness(E)\exactness(F)\,\|u\|_{tb}.
\end{equation}
Moreover, if $E,F$ are finite dimensional then for any $\eps>0$, $n \ge 1$, and $d\ge  128\eps^{-2}\ln (8n/\eps)$, 
$$(1-\eps)\nc{u}\,\leq\, 4 \exactness_n(E)\exactness_n(F)\,\|u_d^\Psi\|,$$
where $\Psi=\Psi_d$ is as defined in~\eqref{eq:def-psi}. 
\end{theorem}

As before, we note that the result from~\cite{JungeP95} is in fact slightly stronger and proves that inequality~\eqref{eq:ps} still holds for the variant of $\nc{u}$ in which the constraint~\eqref{eq:sdpcons} is replaced by~\eqref{eq:l1-cons} (with $t_i=1$).

\begin{proof}[Proof of Corollary~\ref{cor:exact-gt}]
The proof follows along the same lines as that of Corollary~\ref{cor:os-gt}. As before, the first inequality is~\eqref{eq:relaxation}. For the second inequality, let $\eps>0$ and $(x_i,y_i,t_i)_i$ satisfying~\eqref{eq:sdpcons} and such that 
\begin{align*}
\Big\|\sum_i u(x_i,y_i)\Big\| \geq (1-\eps) \os{u}.
\end{align*}
As in the proof of Corollary~\ref{cor:os-gt}, by Theorem~\ref{thm:os-gt} there exists sequences $(\tilde{x}_j)$, $(\tilde{y}_j)$, and for any $d$ a unit vector $\Phi\in\C^d\otimes \C^d$ such that 
\begin{align*}
\Big| \sum_{j} \,u_d^{\Phi}(\tilde{x}_{j},\tilde{y}_{j}) \Big| 
&\geq \Big(1-\eps- C\,\frac{\ln\big(1+\max_i\{t_i,t_i^{-1}\}\big)}{1+\ln d}\Big)\,\os{u}.
\end{align*}
By choosing $d \ge (1+\max_i\{t_i,t_i^{-1}\})^{C / \eps}$, we obtain
\begin{equation*}
(1-2\eps)\os{u}\leq\nc{u_d^\Phi}\leq  4\exactness(E)\exactness(F)\,\|u_d^\Phi\|_{tb}\leq 4\exactness(E)\exactness(F)\|u\|_{jcb},
\end{equation*}
 where the first inequality holds since by~\eqref{eq:os-gt-cons} the $(\tilde{x}_i,\tilde{y}_i,t_i=1)$ satisfy~\eqref{eq:sdpcons}, the second follows from applying Theorem~\ref{thm:jungepisier} to $u_d^\Phi : E\otimes M_d \times F\otimes M_d\to \C$ (and using that for any $d$ it holds that $\exactness(E\otimes M_d)\leq\exactness(E)$, and similarly for $F$), and the third inequality follows from~\eqref{eq:uphi-norm}. We complete the proof by letting $\eps\to 0$.

For the ``moreover" part of the corollary, using the quantitative statement in Theorem~\ref{thm:jungepisier}, for any $d\geq 1$, if $d'\geq  128\eps^{-2}\ln (8nd/\eps)$ then 
$$(1-\eps)\nc{u_d^\Phi}\,\leq\, 4\exactness_{nd}(E\otimes M_d)\exactness_{nd}(F\otimes M_d)\,\big\|u_{dd'}^{\Phi\otimes \Psi}\big\|\,\leq\, 4 \exactness_n(E)\exactness_n(F) \big\|u_{dd'}^{\Phi\otimes \Psi}\big\|.$$
Claim~\ref{claim:maxt} shows that we can choose the sequence $(x_i,y_i,t_i)_{i}$ such that $\max_i \big\{ t_i,t_i^{-1} \big\} \leq 8\,\rc(E)\rc(F)/\eps$. Together with the bound on $d$ shown above, we obtain the estimate claimed in the corollary.  
\end{proof}

\appendix

\section{Omitted proofs}\label{app:additional}

\subsection{Upper bounds on norms}\label{app:ub-norms}

Let $\cA,\cB$ be $C^*$-algebras, $E\subseteq\cA$, $F\subseteq\cB$ operator spaces, and $u:E \times F \to \C$ a bilinear form. In this section we prove the inequalities
\begin{align}\label{eq:normbounds}
 \|u\|_{tb} \,\leq\, \nc{u} ~~~\mbox{and}~~~ \|u\|_{jcb} \,\leq\, \os{u},
\end{align}
starting with the second one.
Let $\eps>0$, and $d\geq 1$ an integer, $\Omega,\Omega'\in\C^d \otimes \C^d$ unit vectors,  $a=\sum a_i \otimes x_i \in E\otimes M_d$, $b=\sum b_i\otimes y_i\in F\otimes M_d$ such that $\|a\|_{\min}\leq 1$, $\|b\|_{\min}\leq 1$, and 
\begin{equation}\label{eq:lb-1}
\Big| \Big\langle\Omega,\,\Big(\sum_{i,j} u(a_i,b_j) \, x_i\otimes y_j\Big)\,\Omega'\Big\rangle \Big| \geq (1-\eps)\|u_d\|.
\end{equation}
Write $\Omega = \sum_i \lambda_i\, e_i\otimes f_i$, $\Omega' = \sum_i \mu_i \, g_i\otimes h_i$, for some orthonormal families $\{e_i\}$, $\{f_i\}$, $\{g_i\}$, $\{h_i\}$ and positive reals $\lambda_i,\mu_i$, and define $t_{i,j} := \frac{\mu_j}{\lambda_i}$,
$$ \tilde{x}_{i,j} \,:=\, \lambda_i\,\sum_k \langle e_i,\, x_k g_j\rangle \,a_k \quad\text{and}\quad \tilde{y}_{i,j} \,:=\, \mu_j \, \sum_k \langle f_i, y_k h_j \rangle\, b_k.$$
Then 
\begin{align*}
\Big\| \sum_{i,j}  \,\tilde{x}_{i,j} \tilde{x}_{i,j}^* \Big\| &= \Big\| \sum_{i,j}  \lambda_i^2\, \Big(\sum_k \langle e_i,\, x_k g_j\rangle\,a_k\Big)\Big(\sum_k \langle g_j,\, x_k^* e_i\rangle\,a_k^*\Big) \Big\|\\
&\leq \sum_{i} \lambda_i^2 \Big\|\Big(\sum_k a_k \otimes x_k\Big)\Big(\sum_k a_k \otimes x_k\Big)^*\Big\|_{\min} \\
&= \big\|a\big\|_{\min}^2 \,\leq\,1.
\end{align*}
Similar bounds can be proven for the three other terms appearing in~\eqref{eq:sdpcons}, so that $(\tilde{x}_{i,j},\tilde{y}_{i,j},t_{i,j})$ satisfies the constraint~\eqref{eq:sdpcons}. One immediately checks from the definition that
$$ \sum_{i,j}\, u(\tilde{x}_{i,j},\tilde{y}_{i,j})\,=\,\Big\langle\Omega,\,\Big(\sum_{i,j} u(a_i,b_j) \cdot x_i\otimes y_j\Big)\,\Omega'\Big\rangle,$$
hence by~\eqref{eq:lb-1} we have $\os{u} \geq (1-\eps)\|u_d\|$. Taking the limit as $\eps\to 0$ and $d\to\infty$ proves the second inequality in~\eqref{eq:normbounds}. For the first it suffices to recall that in the tracially bounded case $\Omega=\Omega'=\Psi_d$, so $\lambda_i = \mu_j = d^{-1/2}$ for every $i,j$, and therefore $t_{i,j}=1$. 

\subsection{Upper bounds on \texorpdfstring{$\rc$}{eta-RC}}\label{sec:eta-bound}

\begin{claim}\label{clm:etacbdist}
For any operator spaces $E$ and $F$, $\rc(E) \le d_{cb}(E,F) \rc(F)$.
\end{claim}
\begin{proof}
For any $\eps>0$, let $v:E\to F$ be such that $\|v\|_{cb} \|v^{-1}\|_{cb} \leq (1+\eps)d_{cb}(E,F)$, and we may assume without loss of generality that $\|v\|_{cb} \leq (1+\eps)d_{cb}(E,F)$ and $\|v^{-1}\|_{cb}\leq 1$. Therefore, for any 
finite sequence $(x_i)_i$ of elements of $E$, we have (see, e.g., Exercise 1.3 in~\cite{pisierbook})
\begin{align*}
\Big\|\sum_i x_i^* x_i \Big\|  \leq \Big\|\sum_i v(x_i)^* v(x_i) \Big\|
&\leq\, \rc(F)^2 \,\Big\|\sum_i v(x_i) v(x_i)^* \Big\|
\leq  \rc(F)^2 ((1+\eps) d_{cb}(E,F))^2 \,\Big\|\sum_i x_i x_i^* \Big\|,
\end{align*}
which together with a symmetric bound on $\|\sum_i x_i x_i^* \|$ and taking the limit $\eps \to 0$ completes the proof. 
\end{proof}

\begin{claim}\label{clm:rcbound}
For any $n \ge 1$, $\rc(M_n)\leq \sqrt{n}$. More generally, $\rc(E)\leq \sqrt{n}\exactness_n(E)$ for any operator
space $E$ for which $\exactness_n(E)<\infty$.
\end{claim}
\begin{proof}
For any finite sequence $(x_i)$ of elements of $M_n$, we have 
\begin{align*}
\Big\|\sum_i x_i^* x_i \Big\|  \leq \Tr\Big( \sum_i x_i^* x_i \Big) 
= \Tr\Big( \sum_i x_i x_i^* \Big) &\leq\, n\,\Big\|\sum_i x_i x_i^* \Big\|
\end{align*}
which together with a symmetric bound on $\|\sum_i x_i x_i^* \|$ proves the first part of the claim.
The second part follows easily from Claim~\ref{clm:etacbdist}.
\end{proof}

The following claim was communicated to us by Gilles Pisier (see~\cite{pisierbook} for the definition of $OH$).  

\begin{claim}\label{clm:oh}
The operator Hilbert space $OH$ satisfies $\rc(OH)=1$. As a result, $\rc(E) \le \sqrt{n}$ for any operator space $E$ of dimension $n$.
\end{claim}
\begin{proof}
By Exercise 7.6 of~\cite{pisierbook}, for any $(x_i)_i$ in $OH$,
\begin{align*}
\Big\|\sum_i x_i^* x_i \Big\| = \Big\|\sum_i x_i x_i^* \Big\| = \Big( \sum_{i,j} |\langle x_i,x_j \rangle|^2 \Big)^{1/2},
\end{align*}
and so we get $\rc(OH)=1$. The second part of the claim follows from Claim~\ref{clm:etacbdist} and the fact that $d_{cb}(E,OH_n) \le \sqrt{n}$
for any $n$-dimensional operator space $E$ (see Corollary 7.7 in~\cite{pisierbook}).
\end{proof}

\subsection{Quantitative version of Theorem~\ref{thm:jungepisier}}\label{apx:quantitative}

The following claim is a direct consequence of the results in~\cite{HT99}.

\begin{claim}\label{claim:ht} Let $n\geq 1$ be an integer, $a_1,\ldots,a_r\in M_n$ and $0<\gamma\leq 1$ such that
$$  \Big\| \sum_j a_j^* a_j \Big\|\,\leq\,\gamma\quad\text{and}\quad \Big\| \sum_j a_ja_j^*  \Big\|\,\leq\, 1.$$ 
For any integer $d$, define 
$$S_d \,:=\, \sum_j a_j \otimes G_j \,\in M_n\otimes M_d,$$
 where for each $j$, $G_j$ is a $d\times d$ complex matrix with entries $(G_j)_{k,\ell} = (g_{jk\ell}+i \,h_{jk\ell})/\sqrt{2}$, where $\{g_{jk\ell},h_{jk\ell}\}$ are distributed as independent real normal random variables with mean $0$ and variance $1/d$. Then for any $0<\eps\leq 1$ and $d \ge 32\eps^{-2}\ln (4n/\eps)$,
$$\textsc{E}\big[\,\| S_d \|^2\,\big]\, \leq\, (1+\eps)\,(\sqrt{\gamma}+1)^2.$$
\end{claim}

\begin{proof}
Eq.~(0.1) from (0.5 Key Estimates) in~\cite{HT99} states that for any $0\leq \tau \leq \min\{d/(2\gamma),d/2\} = d/2$ it holds that 
\begin{align*}
\textsc{E}\big[e^{\tau S_d^* S_d}\big]\,\leq\, e^{(\sqrt{\gamma}+1)^2 \tau + (\gamma+1)^2 \tau^2/d} \Id_{M_{nd}}.
\end{align*}
Taking the trace on both sides, we get
$$\textsc{E}\big[\big\|e^{\tau{S}_d^* {S}_d}\big\|\big]\,\leq\,\textsc{E}\big[\Tr\big(e^{\tau{S}_d^* {S}_d}\big)\big]\,\leq\, nd\, e^{(\sqrt{\gamma}+1)^2 \tau + (\gamma+1)^2 \tau^2/d}.$$
Using $\|e^{\tau{S}_d^* {S}_d}\| = e^{\tau\|{S}_d {S}_d^*\|}$ and concavity of the logarithm, for $\tau>0$
$$\textsc{E}\big[\big\|{S}_d^* {S}_d\big\|\big]\,\leq\, (\sqrt{\gamma}+1)^2 + \frac{(\gamma+1)^2 \tau}{d} + \frac{\ln(nd)}{\tau}.$$
By setting $\tau = \eps d/2$ we get
$$ \frac{(\gamma+1)^2 \tau}{d} + \frac{\ln(nd)}{\tau}\,\leq\, (\sqrt{\gamma}+1)^2 \frac{\eps}{2} + \frac{2\ln(nd)}{\eps d}  \,\leq\, \eps\, (\sqrt{\gamma}+1)^2$$
provided $d\geq (4/\eps^2)\ln (nd)$, which is guaranteed by the lower bound on $d$ placed in the claim.
\end{proof} 

As an immediate corollary we obtain the following. 

\begin{corollary}\label{cor:ht} Let $\cA$ be a $C^*$-algebra and $E\subseteq \cA$ a finite-dimensional operator space.
 Let $(a_i)_i$ be a finite sequence of elements of $E$, $d$ an integer, and $\gamma,S_d$ be as in Claim~\ref{claim:ht}. Then for any $0<\eps\leq 1$, integer $n\geq 1$ such that $\exactness_n(E)<\infty$, and $d\geq 32\eps^{-2}\ln (4n/\eps)$,
$$\textsc{E}\big[\,\| S_d \|^2\,\big]\, \leq\, (1+\eps)\,\exactness_n(E)^2\,(\sqrt{\gamma}+1)^2.$$
\end{corollary}

\begin{proof}
By definition of $\exactness_n(E)$, there exists a completely bounded isomorphism $v:E\to F\subseteq M_n$ such that $\|v\|_{cb} \|v^{-1}\|_{cb} = \exactness_n(E)$, and we may assume without loss of generality that $\|v\|_{cb} = 1$ and $\|v^{-1}\|_{cb}=\exactness_n(E)$. Since $\|v\|_{cb}\leq 1$, the elements $v(a_i)\in M_n$ satisfy (see, e.g., Exercise 1.3 in~\cite{pisierbook})
$$\Big\| \sum_i v(a_i)^* v(a_i) \Big\|\,\leq\, \Big\| \sum_i a_i^* a_i \Big\|\,\leq\,\gamma\quad\text{and}\quad  \Big\| \sum_i v(a_i)v(a_i)^*  \Big\|\,\leq\, \Big\| \sum_i a_i a_i^*  \Big\|\,\leq\,1.$$
Applying Claim~\ref{claim:ht} to the $v(a_i)$, we obtain that
for any $\eps>0$ and $d\geq 32\eps^{-2}\ln (4n/\eps)$,
$$\textsc{E}\big[\,\big\| \sum_i v(a_i) \otimes G_i \big\|^2\,\big]\, \leq\, (1+\eps)\,(\sqrt{\gamma}+1)^2.$$
Using $\|v^{-1}\otimes \Id_{M_d}\|\leq\|v^{-1}\|_{cb}\leq \exactness_n(E)$ proves the corollary. 
\end{proof}

Using Corollary~\ref{cor:ht}, we can prove the quantitative part of Theorem~\ref{thm:jungepisier}. Here we are essentially following the proof given in~\cite[Section 16]{PisierGT}, but while keeping track of the parameters.

\begin{proof}[Proof (of quantitative part of Theorem~\ref{thm:jungepisier})]
We prove the quantitative part using the original stronger form of Theorem~\ref{thm:jungepisier}, i.e., 
with the constraint~\eqref{eq:sdpcons} in the definition of $\nc{u}$ replaced by~\eqref{eq:l1-cons} (with $t_i=1$). Let $(x_i,y_i)_i$ be such that 
$$ \Big| \sum_i \, u(x_i,y_i)\Big|  \,\geq\, (1-\eps/2) \nc{u},$$
and the sequence $(x_i,y_i,t_i=1)_i$ satisfies the constraint~\eqref{eq:l1-cons}. Let $d$ be such that $d> 128\eps^{-2}\ln (8n/\eps)$, and for every $i$ let $G_i$ be a $d\times d$ matrix with independent entries distributed as in the statement of Claim~\ref{claim:ht}. Define 
$$ x = \sum_i {x}_i \otimes G_i\quad \text{and}\quad y = \sum_i {y}_i \otimes \overline{G_i},$$
where $\overline{G_i}$ denotes the entrywise complex conjugate, and note that by Corollary~\ref{cor:ht} our choice of $d$ together with the constraint~\eqref{eq:l1-cons} implies that
\begin{equation}\label{eq:obj-val-1}
 \textsc{E}\big[\, \|x\|\|y\| \,\big] \,\leq\, \Big(\textsc{E}\big[\|x\|^2\big]\, \textsc{E}\big[ \|y\|^2 \big]\Big)^{1/2} \,\leq\,   4\exactness_n(E)\exactness_n(F) (1+\eps/2).
\end{equation}
We may also compute 
\begin{align*}
\Big|\textsc{E}\big[ \,u_d^\Psi(x,y)\,\big]\Big|\,=\,\Big|\textsc{E}\Big[\, \sum_{i,j} \, d^{-1} \Tr(G_iG_j^*) \,u({x}_i,{y}_j)\Big]\Big| \,=\,  \Big|\sum_i \, u({x}_i,{y}_i)\Big|\,\geq\, (1-\eps/2)\nc{u},
\end{align*}
which using $|u_d^\Psi(x,y)|\leq \|u_d^\Psi\| \|x\|\|y\|$ for any $x,y$ together with~\eqref{eq:obj-val-1} completes the proof. 
\end{proof}

\bibliographystyle{alphaabbrvprelim}
\bibliography{oss}

\begin{thebibliography}{vDH03}
\expandafter\ifx\csname urlstyle\endcsname\relax
  \providecommand{\doi}[1]{doi:\discretionary{}{}{}#1}\else
  \providecommand{\doi}{doi:\discretionary{}{}{}\begingroup
  \urlstyle{rm}\Url}\fi

\bibitem[Ble89]{Blecher89}
D.~P. Blecher.
\newblock Tracially completely bounded multilinear maps on ${C}^*$-algebras.
\newblock \emph{Journal of the London Mathematical Society}, s2-39(3):514--524,
  1989.

\bibitem[Ble92]{Blecher92}
D.~P. Blecher.
\newblock {Generalizing Grothendieck's program}.
\newblock In K.~Jarosz, editor, \emph{Lecture Notes in Pure and Applied Math.},
  volume 136. CRC Press, 1992.

\bibitem[ER91]{ER91}
E.~Effros and Z.~J. Ruan.
\newblock A new approach to operator spaces.
\newblock \emph{Canadian Math. Bull.}, 34:329--337, 1991.

\bibitem[Gro53]{Gro53}
A.~Grothendieck.
\newblock R\'esum\'e de la th\'eorie m\'etrique des produits tensoriels
  topologiques.
\newblock \emph{Bol. Soc. Mat. S\~ao Paulo}, 8:1--79, 1953.

\bibitem[Haa85]{Haagerup85NCGT}
U.~Haagerup.
\newblock The {G}rothendieck inequality for bilinear forms on ${C}^*$-algebras.
\newblock \emph{Advances in Mathematics}, 56(2):93 -- 116, 1985.

\bibitem[HM08]{HM08}
U.~Haagerup and M.~Musat.
\newblock The {E}ffros-{R}uan conjecture for bilinear forms on
  ${C}^*$-algebras.
\newblock \emph{Inventiones Mathematicae}, 174:139--163, 2008.

\bibitem[HT98]{HT99}
U.~Haagerup and S.~Thorbj{\o}rnsen.
\newblock Random matrices and {K}-theory for exact {C}*-algebras.
\newblock \emph{Doc. Math.}, 4:341--450, 1998.

\bibitem[JP95]{JungeP95}
M.~Junge and G.~Pisier.
\newblock Bilinear forms on exact operator spaces and {$B(H)\otimes B(H)$}.
\newblock \emph{Geom. Funct. Anal.}, 5(2):329--363, 1995.

\bibitem[Pis78]{Pisier78NCGT}
G.~Pisier.
\newblock Grothendieck's theorem for noncommutative {$C^{\ast} $}-algebras,
  with an appendix on {G}rothendieck's constants.
\newblock \emph{J. Funct. Anal.}, 29(3):397--415, 1978.

\bibitem[Pis03]{pisierbook}
G.~Pisier.
\newblock \emph{Introduction to Operator Space Theory}.
\newblock Cambridge University Press, 2003.

\bibitem[Pis12]{PisierGT}
G.~Pisier.
\newblock Grothendieck's theorem, past and present.
\newblock \emph{Bull. Amer. Math. Soc.}, 49(2):237--323, 2012.
\newblock Also available at arXiv:1101.4195.

\bibitem[PS02]{PS02OSGT}
G.~Pisier and D.~Shlyakhtenko.
\newblock Grothendieck's theorem for operator spaces.
\newblock \emph{Inventiones Mathematicae}, 150:185--217, 2002.

\bibitem[RV12]{RegevV12a}
O.~Regev and T.~Vidick.
\newblock Quantum {XOR} games, 2012.
\newblock In preparation.

\bibitem[vDH03]{vDH03}
W.~van Dam and P.~Hayden.
\newblock Universal entanglement transformations without communication.
\newblock \emph{Phys. Rev. A}, 67:060302, Jun 2003.

\end{thebibliography}

\end{document}